\documentclass[12pt,article,reqno]{amsart}
\usepackage[margin=1in]{geometry}

\usepackage{amsfonts,amsthm,latexsym,amsmath,amssymb,amscd,amsmath, epsf, breqn}
\usepackage[utf8]{inputenc}
\usepackage{url} 
\usepackage{hyperref}
\usepackage{graphicx,epsfig,color} 
\usepackage[enableskew]{youngtab}
\usepackage{ytableau,tikz,varwidth}

\usepackage{tikz,graphicx}
\usepackage{tikz-cd}
\usepackage{tikz-3dplot}
\usetikzlibrary{calc}
\usetikzlibrary{arrows}
\usetikzlibrary{shapes}
\usetikzlibrary{patterns}
\usetikzlibrary{positioning}
\usetikzlibrary{arrows.meta}
\usetikzlibrary{decorations.markings}
\usepackage{epstopdf}
\newcommand\hidefigure[1]{#1}

\newtheorem{definition}{Definition}

\newtheorem{theorem}{Theorem}
\newtheorem{lemma}{Lemma}
\newtheorem{corollary}{Corollary}
\newtheorem{proposition}{Proposition}

\newtheorem{remark} {Remark}
\newtheorem{example} {Example}
\newtheorem{problem} {Problem}

\newcommand {\bC} {\mathbb {C}}

\newcommand {\bN} {\mathbb {N}} 

\newcommand {\bQ} {\mathbb {Q}}
\newcommand {\bR} {\mathbb {R}}
\newcommand {\bY} {\mathbb {Y}}
\newcommand {\bZ} {\mathbb {Z}}

\newcommand{\F} {\mathcal{F}}

\newcommand{\R} {\mathcal{R}}

\newcommand{\Fl}{\mathcal{F}\ell}

\font\germ=eufm10
\newcommand{\Sch} {\hbox{\germ S}}
\newcommand{\Schur} {s}
\newcommand{\Schurp} {{\bf \germ s}}
\newcommand{\BSch} {\overleftarrow{\hbox{\germ S}}}
\newcommand{\RC}{\mathcal{RC}}

\newcommand{\Diag} {\mathcal{Y}}
\newcommand{\VDiag} {\bQ\Diag}
\newcommand{\mydiff} {\xi}

\newcommand{\PermutN}{S_{\bN}} 
\newcommand{\PermutZ}{S_{\bZ}}
\newcommand{\shift}{\tau}
\newcommand{\mybinom}[2]{\Bigl(\begin{array}{@{}c@{}}#1\\#2\end{array}\Bigr)}

\newcommand{\Inter}{\scalebox{0.5}{\tikz{
\draw[step=1cm,gray,very thin] (0,1) grid (1,2);
\draw[line width=0.5mm] (0.5,1) -- (0.5,2);
\draw[line width=0.5mm] (0,1.5) -- (1,1.5);
\draw [fill=red] (0.5,1.5) circle [radius=0.1];}}}

\newcommand{\NoInter}{\scalebox{0.5}{\tikz{
\draw[step=1cm,gray,very thin] (3,1) grid (4,2);
\draw[line width=0.5mm] (3,1.5) -- (3.5,2);
\draw[line width=0.5mm] (3.5,1) -- (4,1.5);}}}

\newcommand{\Comment}[1]{}

\newcommand{\WTW}[1]{}
\newcommand{\nextsection}{\medskip}
\newcommand{\nextsubsection}{\smallskip}
\newcommand{\CiteYes}{}

\DeclareUnicodeCharacter{2212}{-}
\begin{document}

          \numberwithin{equation}{section}

 \title[Differential operators on Schur and Schubert polynomials]
          {Differential operators on Schur and Schubert polynomials}

\author[Gleb Nenashev]{Gleb Nenashev}
\address{
Department of Mathematics, Massachusetts Institute of Technology,
77 Massachusetts Avenue
Cambridge, MA 02139, US}
\email{nenashev@mit.edu}
\thanks{This research is supported by the Knut and Alice Wallenberg Foundation (KAW2017.0394)}
\subjclass[2010]{05E05 and 14M15} 
\keywords{Schubert polynomials, Schur functions, Structure constants, Bosonic operators,  Reduced decompositions}

\begin{abstract}
This paper deals with decreasing operators on back stable Schubert polynomials. We study two operators $\xi$ and $\nabla$ of degree $-1$, which satisfy the Leibniz rule. Furthermore, we show that all other such operators are linear combinations of $\xi$ and $\nabla$. 

For the case of Schur functions, these two operators  fully determine the product of Schur functions, i.e., it is possible to define the Littlewood-Richardson coefficients only from $\xi$ and $\nabla$. This new point of view on Schur functions gives us an elementary proof of the Giambelli identity and of Jacobi-Trudi identities.

For the case of Schubert polynomials, we construct a bigger class of decreasing operators as expressions in terms of $\xi$ and $\nabla$, which are indexed by Young diagrams. Surprisingly, these operators are related to Stanley symmetric functions. In particular, we extend bosonic operators from Schur to Schubert polynomials.
\end{abstract}

\maketitle

\setcounter{tocdepth}{1}
\tableofcontents

\section{Introduction}
  
Schubert calculus is a branch of algebraic geometry introduced in the nineteenth century by Herman Schubert, in order to solve various counting problems of enumerative geometry.   
Some of the key objects of the theory are polynomials in many variables of a specific form, called Schubert polynomials. 
  Schubert polynomials were defined by A.\,Lascoux and M.\,P.\,Sch\"utzenberg in~1982~\cite{LS1,LS2}, see also the book~\cite{Mac}. Prior to their works it was almost considered by I.\,N.\,Bernstein, I.\,M.\,Gelfand, and S.\,I.\,Gelfand~\cite{BGG} and M.\,Demazure~\cite{Dem}, where they gave a description of the cohomology ring of the complete flag variety~$\Fl_n$.
These polynomials represent cohomology classes of Schubert cycles in flag varieties.
Schubert polynomials have been actively studied for the last 30 years. The famous Schur functions are a specific case of Schubert polynomials: they correspond to the so-called Grassmannian permutations.

In this paper we study decreasing operators. The main two operators are $\mydiff$ and $\nabla$, they have degree $-1$. These two operators are well-defined for Schur polynomials and for generalizations of Schubert polynomials (the so-called back stable Schubert polynomials). Below in the introduction, we wrote formulas for these operators.

For the case of Schur polynomials, these operators are related to Kerov's operators (see some properties of ``down'' operator $\mydiff +z \nabla$ in~\cite{Oko}). For the case of Schubert polynomials the operator $\nabla$ was studied in~\cite{HPSW,St-nabla}. The main result of our paper is that these two operators determine the product when both factors are Schur, and almost determines it when one factor is Schur and the other Schubert. In particular, differential operators provide a new elementary proof of the Giambelli's formula and both Jacobi-Trudi formulas for Schur polynomials. They present one more connection between Murnaghan-Nakayama rule and characters of symmetric group, and some other properties. For the case of Schubert polynomials, these two operators give a new algorithm for the product, which is asymptotically the fastest. We will express bosonic operators in terms of  $\mydiff$ and $\nabla$ for Schur polynomials, which extends the notion to Schubert polynomials. We also construct more decreasing operators indexed by Young diagrams.

\nextsubsection
\subsection{Operations on Young Diagrams} 
Let $\Diag$ be  the set of Young diagrams (partitions), i.e., $\lambda=(\lambda_1,\ldots, \lambda_k)\in\Diag$ if and only if $\lambda_i$ are non-negative integers and are weakly decreasing. 
Consider the vector space $\VDiag$ consisting of formal finite sums of $\Diag$ with rational coefficients, i.e.,

$$\VDiag:=\left\{\sum_{i=1}^k a_i\lambda^{(i)}:\ k\in \bN,\ a_i\in \bQ,\ \lambda^{(i)}\in \Diag \right\}.$$
We denote by $\VDiag_n$ the subspace generated by diagrams with exactly $n$ boxes. We also assume the empty diagram  belongs to $\Diag$; it will be denoted by unit $1$, i.e., $\VDiag_0=\bQ$.

Now we present two linear operators $\mydiff$ and $\nabla$ on $\VDiag$. Namely, for a Young diagram $\lambda\in \Diag$, we have
$$\mydiff(\lambda)=\sum_{\substack{(i,j)\in \bN^2\\ \lambda'=\lambda-(i,j)\in \Diag}} \lambda';$$
and
$$\nabla(\lambda)=\sum_{\substack{(i,j)\in \bN^2\\ \lambda'=\lambda-(i,j)\in \Diag}} (j-i)\lambda',$$
i.e.,  both summations are taken over all diagrams after deleting one box (see the example below). 
Since both operators $\mydiff$ and $\nabla$ decrease the number of boxes by one, we will call them differential operators. For the box $(i,j),$ the value $j-i$ is called the content. Note that, for the empty diagram, we have $\mydiff(1)=\nabla(1)=0$; we will thus say that the coefficient of the empty diagram is a constant term.

\begin{example} For the partition $(4,3,1)$ we have the following identities:
  \ytableausetup{smalltableaux}
$$\mydiff\left(\ydiagram{4,3,1}\right)=\ydiagram{3,3,1}+\ydiagram{4,2,1}+\ydiagram{4,3,0}$$
and
$$\nabla\left(\ytableaushort{0123,{ \text{-}1}01,{ \text{-}2}}\right)=3\cdot \ytableaushort{012,{ \text{-}1}01,{ \text{-}2}}+1\cdot \ytableaushort{0123,{ \text{-}1}0,{ \text{-}2}} -2 \cdot \ytableaushort{0123,{ \text{-}1}01,\none}$$
(we put the content inside the boxes).
\end{example}

We have the following key lemma:
\begin{lemma} 
\label{lemma:key} An element from $\VDiag$ is constant if and only if both operators evaluated at it give zero, i.e., $$x\in \bQ\ \ \Longleftrightarrow \  \ \mydiff(x)=\nabla(x)= 0.$$
\end{lemma}

We say that $f: \Diag^2\to \Diag$ is the multiplication map if
\begin{itemize}
\item $f$ satisfies the distributive property
\item for $n,m\in \bN$ and $x\in \VDiag_n,y\in \VDiag_m,$ $f(x,y)\in \VDiag_{(n+m)}$;
\item for $a,b\in \bQ,$ $f(a,b)=ab$;
\item for any $x,y\in \VDiag$, $\mydiff(f(x,y))=f(\mydiff(x),y)+f(x,\mydiff(y))$;
\item for any $x,y\in \VDiag$, $\nabla (f(y,x))=f(\nabla (x),y)+f(x,\nabla (y))$.
\end{itemize}
The last two properties say that $f$ satisfies the Leibniz rule for $\mydiff$ and $\nabla$, i.e., they are differential operators. The above properties are sufficient to define the multiplication map, since we already have the following easy corollary from the key lemma:

\begin{corollary}
There is at most one multiplication map.	
\end{corollary}

Moreover, such a multiplication map does exist:

\begin{theorem}
There is a unique multiplication map. 

Furthermore, it is commutative and associative; it is given by:
$$\lambda * \mu =\sum_{\nu} c_{\lambda,\mu}^{\nu} \nu,$$
where the $c_{\lambda,\mu}^{\nu}$ are the Littlewood-Richardson coefficients.
\end{theorem}

  The theorem above says that the differential operators $\mydiff$ and $\nabla$ determine the Littlewood-Richardson coefficients. 
This means that we can work with Schur functions without ``functions/polynomials'' and without the cohomology ring, which is a new point of view on them. 
We will denote this ring by $\bY$. 
In particular, the differential operators provide a new elementary proof of the Giambelli's formula and both Jacobi-Trudi formulas for Schur polynomials.

These two operators came from studying Schubert polynomials, which are generalizations of Schur functions.   
\nextsubsection
\subsection{Schubert polynomials} 

It is easy to define Schubert polynomials recursively using the divided differences operators
$$\partial_i f: =\frac{f-s_if}{x_i-x_{i+1}}.$$

\begin{definition} For a permutation $w_0=(n,n-1,\ldots,1)\in S_n$, we define its Schubert polynomial as
$$\Sch_{w_0}=x_1^{n-1}x_2^{n-1}\cdots x_{n-1}^{1}\in \bR[x_1,x_2,\ldots].$$
For each permutation $w\in S_n$, its Schubert polynomial is given by 
$$\Sch_w=\partial_{i_1}\cdots \partial_{i_{k}} \Sch_{w_0},$$
where $s_{i_1}\ldots s_{i_k}=w^{-1}w_0$ is a reduced decomposition of $w^{-1}w^0,$ i.e., 
$k=\ell(w^{-1}w_0)=\mybinom{n}{2}-\ell(w)$.
\end{definition}

This polynomials are well defined, i.e., they are independent of the choice of a reduced decomposition. Even more is true;
define $\PermutN$ as the set of all permutations of $\bN=\{1,2,3,4,\ldots \}$ fixing all but finitely many elements. We have the natural inclusions:
$$S_0= S_1\subset S_2 \subset S_3 \subset \ldots \subset \PermutN.$$

\begin{theorem}[cf~\cite{LS1,LS2}]
Schubert polynomials are well-defined for $w\in \PermutN$ and they form a linear basis of $\bR[x_1,x_2,x_3,\ldots]$.
\end{theorem}

Since they form a linear basis, there are unique structure constants $c_{u,v}^w$ such that
$$\Sch_u\Sch_v=\sum_{w\in\PermutN}c_{u,v}^w\Sch_w\ \textrm{for any}\ u,v\in \PermutN.$$
These structure constants can be seen as a generalization of the
Littlewood-Richardson coefficients. 
There are a lot of well-known descriptions of these coefficients; see for example,~\cite{BZ,EG,KT,KTW,LR-rule}.

\begin{problem}
The coefficients $c_{u,v}^w,\ u,v,w\in \PermutN$ are $3$-point Gromow-Witten invariants of genus $0$ for flag varieties, which are non-negative integers by representation-theoretical reasons.
An outstanding algebraic combinatorics problem is to give a combinatorial rule for these numbers. There is not even any non-geometric proof of non-negativity. 
\end{problem}

In some particular cases combinatorial rules for multiplication were found. Monk's rule~\cite{Monk} applies when one of the permutations is a simple transposition, see e.g.~\cite{Billey&Co}. Later Pieri's rule and a more general rule for rim hooks were given by F.\,Sottile in~1996~\cite{Sot}, see also~\cite{FL}. K.\,M\'esz\'aros et al. In 2014~\cite{MPP} rewrote  and gave a new proof of the rule for rim hooks (and proved that this way works for hooks with an extra square) in terms of the Fomin-Kirillov algebra~\cite{FK-algebra}. Some other rules with restrictions on both permutations were presented by I.\,Coskun~in~2009~\cite{Cos} and M.\,Kogan~in~2001~\cite{Ko}. Furthermore, A.\,Morrison and F.\,Sottile found an analogue of Murhaghan-Nakayama rule for Schubert polynomials, see~\cite{MS}.  

\medskip

We will work mostly with back stable Schubert polynomials, which are defined for all permutations of integers fixing all but a finite number of elements (denote by $\PermutZ$). A back stable Schubert polynomial is similar to a Schubert polynomial, but has some extra properties (see the definition in section \S3). For these polynomials operators $\mydiff$ and $\nabla$ are given by
$$\mydiff \BSch_u=\sum_{k\in \bZ:\ \ell(s_ku)=\ell(u)-1}\BSch_{s_ku}$$
and
$$\nabla \BSch_u=\sum_{k\in \bZ:\ \ell(s_ku)=\ell(u)-1}k\BSch_{s_ku}.$$
The operator $\nabla$ was defined by R.\,Stanley~\cite{St-nabla} and another interpretation was given by Z.\,Hamaker, O.\,Pechenik, D.\,E.\,Speyer, and A.\,Weigandt~\cite{HPSW}.
Unfortunately, the operator $\mydiff$ cannot be defined for Schubert polynomials.

\nextsubsection
\subsection*{The structure of the paper} 
In the above subsection we defined the ring $\bY$ and presented the definition of Schubert polynomials. 

In \S2 we provide some necessary background to work with Schubert polynomials and Stanley symmetric functions. In \S3 we define back stable Schubert polynomials and introduce the operators $\mydiff$ and $\nabla$ for them. 

In subsection \S4.1 we restrict $\mydiff$ and $\nabla$ to Schur functions. In subsection \S4.2 we prove the main result from introduction~\S1.1. In subsections \S4.3 and \S4.4 we present applications of our new definitions of Schur polynomials. In particular, we present an elementary proof of the determinantal formulas and 
the dual Murnaghan-Nakayama rule.

In section \S5, we return to discussing back stable Schubert polynomials. 

In the last section~\S6, we show how our theory can be used for the case of the product when one factor is a Schubert polynominal and the other is a Schur polynomial. In particular, we show that the introduced operators determine the Gromov-Witten invariants in this case.

\nextsection
\section{Schubert polynomials}
\label{sec:background}
Define $R(w)$ as the set of reduced decompositions of $w\in \PermutN$, and $\ell(w)$ as their corresponding length.

\begin{figure}[htb!]
\scalebox{0.4}{
\hidefigure{
\begin{tikzpicture}

\draw[line width=1mm] (0.5,-0.5) --(0.5,6.5);
\node[text=black] at (0,6) {3};
\node[text=black] at (0,5) {2};
\node[text=black] at (0,4) {7};
\node[text=black] at (0,3) {1};
\node[text=black] at (0,2) {5};
\node[text=black] at (0,1) {4};
\node[text=black] at (0,0) {6};
\draw[line width=1mm] (8.5,-0.5) --(8.5,6.5);
\node[text=black] at (9,6) {1};
\node[text=black] at (9,5) {2};
\node[text=black] at (9,4) {3};
\node[text=black] at (9,3) {4};
\node[text=black] at (9,2) {5};
\node[text=black] at (9,1) {6};
\node[text=black] at (9,0) {7};
\draw[line width=0.5mm] (7.5,5) -- (8.5,6);
\draw[line width=0.5mm] (7.5,6) -- (8.5,5);
\draw[line width=1mm, violet] (7.6,5.1) -- (8.4,5.9);
\draw[line width=1mm, violet] (7.6,5.9) -- (8.4,5.1);
\draw[line width=0.5mm] (8.5,4) -- (7.5,4);
\draw[line width=0.5mm] (8.5,3) -- (7.5,3);
\draw[line width=0.5mm] (8.5,2) -- (7.5,2);
\draw[line width=0.5mm] (8.5,1) -- (7.5,1);
\draw[line width=0.5mm] (8.5,0) -- (7.5,0);
\draw[line width=0.5mm] (6.5,4) -- (7.5,5);
\draw[line width=0.5mm] (6.5,5) -- (7.5,4);
\draw[line width=1mm, violet] (6.6,4.1) -- (7.4,4.9);
\draw[line width=1mm, violet] (6.6,4.9) -- (7.4,4.1);
\draw[line width=0.5mm] (7.5,6) -- (6.5,6);
\draw[line width=0.5mm] (7.5,3) -- (6.5,3);
\draw[line width=0.5mm] (7.5,2) -- (6.5,2);
\draw[line width=0.5mm] (7.5,1) -- (6.5,1);
\draw[line width=0.5mm] (7.5,0) -- (6.5,0);
\draw[line width=0.5mm] (5.5,5) -- (6.5,6);
\draw[line width=0.5mm] (5.5,6) -- (6.5,5);
\draw[line width=1mm, violet] (5.6,5.1) -- (6.4,5.9);
\draw[line width=1mm, violet] (5.6,5.9) -- (6.4,5.1);
\draw[line width=0.5mm] (6.5,4) -- (5.5,4);
\draw[line width=0.5mm] (6.5,3) -- (5.5,3);
\draw[line width=0.5mm] (6.5,2) -- (5.5,2);
\draw[line width=0.5mm] (6.5,1) -- (5.5,1);
\draw[line width=0.5mm] (6.5,0) -- (5.5,0);
\draw[line width=0.5mm] (4.5,2) -- (5.5,3);
\draw[line width=0.5mm] (4.5,3) -- (5.5,2);
\draw[line width=1mm, violet] (4.6,2.1) -- (5.4,2.9);
\draw[line width=1mm, violet] (4.6,2.9) -- (5.4,2.1);
\draw[line width=0.5mm] (5.5,6) -- (4.5,6);
\draw[line width=0.5mm] (5.5,5) -- (4.5,5);
\draw[line width=0.5mm] (5.5,4) -- (4.5,4);
\draw[line width=0.5mm] (5.5,1) -- (4.5,1);
\draw[line width=0.5mm] (5.5,0) -- (4.5,0);
\draw[line width=0.5mm] (3.5,0) -- (4.5,1);
\draw[line width=0.5mm] (3.5,1) -- (4.5,0);
\draw[line width=1mm, violet] (3.6,0.1) -- (4.4,0.9);
\draw[line width=1mm, violet] (3.6,0.9) -- (4.4,0.1);
\draw[line width=0.5mm] (4.5,6) -- (3.5,6);
\draw[line width=0.5mm] (4.5,5) -- (3.5,5);
\draw[line width=0.5mm] (4.5,4) -- (3.5,4);
\draw[line width=0.5mm] (4.5,3) -- (3.5,3);
\draw[line width=0.5mm] (4.5,2) -- (3.5,2);
\draw[line width=0.5mm] (2.5,1) -- (3.5,2);
\draw[line width=0.5mm] (2.5,2) -- (3.5,1);
\draw[line width=1mm, violet] (2.6,1.1) -- (3.4,1.9);
\draw[line width=1mm, violet] (2.6,1.9) -- (3.4,1.1);
\draw[line width=0.5mm] (3.5,6) -- (2.5,6);
\draw[line width=0.5mm] (3.5,5) -- (2.5,5);
\draw[line width=0.5mm] (3.5,4) -- (2.5,4);
\draw[line width=0.5mm] (3.5,3) -- (2.5,3);
\draw[line width=0.5mm] (3.5,0) -- (2.5,0);
\draw[line width=0.5mm] (1.5,2) -- (2.5,3);
\draw[line width=0.5mm] (1.5,3) -- (2.5,2);
\draw[line width=1mm, violet] (1.6,2.1) -- (2.4,2.9);
\draw[line width=1mm, violet] (1.6,2.9) -- (2.4,2.1);
\draw[line width=0.5mm] (2.5,6) -- (1.5,6);
\draw[line width=0.5mm] (2.5,5) -- (1.5,5);
\draw[line width=0.5mm] (2.5,4) -- (1.5,4);
\draw[line width=0.5mm] (2.5,1) -- (1.5,1);
\draw[line width=0.5mm] (2.5,0) -- (1.5,0);
\draw[line width=0.5mm] (0.5,3) -- (1.5,4);
\draw[line width=0.5mm] (0.5,4) -- (1.5,3);
\draw[line width=1mm, violet] (0.6,3.1) -- (1.4,3.9);
\draw[line width=1mm, violet] (0.6,3.9) -- (1.4,3.1);
\draw[line width=0.5mm] (1.5,6) -- (0.5,6);
\draw[line width=0.5mm] (1.5,5) -- (0.5,5);
\draw[line width=0.5mm] (1.5,2) -- (0.5,2);
\draw[line width=0.5mm] (1.5,1) -- (0.5,1);
\draw[line width=0.5mm] (1.5,0) -- (0.5,0);
\end{tikzpicture}}}
\caption{A reduced decomposition $s_1s_2s_1s_4s_6s_5s_4s_3$ of $3271546\in \PermutZ$.}
\label{red-decomp}
\end{figure}

\begin{definition}
For a reduced decomposition $h=(h_1,h_2,\ldots,h_{\ell(w)})\in R(w).$
Let $C(h)$ be the set of all $\ell(w)$-tuples $(\alpha_1,\ldots,\alpha_{\ell(w)})$ of positive integers such that
\begin{itemize}
\item $1\leq \alpha_1\leq \alpha_2\leq \ldots \leq \alpha_{\ell(w)}$;
\item $\alpha_j\leq h_j$;
\item $\alpha_j< \alpha_{j+1}\ \textrm{if}\ h_j < h_{j+1}$.
\end{itemize}
\end{definition}

\begin{theorem}[cf.~\cite{Billey&Co,FSt}] \label{stanley-formula}
For any permutation $w\in \PermutN$, its Schubert polynomial is given by
$$\Sch_w=\sum_{h\in R(w)} \sum_{\alpha\in C(h)} x_{\alpha_1}x_{\alpha_2}\cdots x_{\alpha_{\ell}}.$$ 
\end{theorem}

There is a well known interpretation of this formula. For a permutation $w\in \PermutN$ we denote by $\RC(w)$  the set of its RC-graphs (pipe dreams). Namely, we have a grid $\{1,2,3,\ldots\}\times \{1,2,3,\ldots\},$ where each square is 
$\Inter$ or $\NoInter$, such that the $i$-th line starts at the left of the box $(1,i)$ and finishes at the top of the box $(w(i),1)$,  and any two lines intersect  at most once (see Figure~\ref{RC-graph}). For an RC-graph we define the monomial $m(w):=x_1^{d_1} x_2^{d_2}x_3^{d_3}\cdots,$ where $d_i$ is the number of intersections in the $i$-th row.
\begin{figure}[htb!]
\scalebox{0.6}{
\hidefigure{
\begin{tikzpicture}
\node[text=black] at (0.5,9.5) {$\omega^{-1}=$};
\node[text=black] at (1.5,9.5) {4};
\node[text=black] at (2.5,9.5) {2};
\node[text=black] at (3.5,9.5) {1};
\node[text=black] at (4.5,9.5) {6};
\node[text=black] at (5.5,9.5) {5};
\node[text=black] at (6.5,9.5) {3};
\node[text=black] at (7.5,9.5) {7};
\node[text=black] at (8.5,9.5) {8};
\node[text=black] at (9.5,9) {\bf .};
\node[text=black] at (9.7,9) {\bf .};
\node[text=black] at (9.9,9) {\bf .};

\node[text=black] at (0.5,8.5) {1};
\node[text=black] at (0.5,7.5) {2};
\node[text=black] at (0.5,6.5) {3};
\node[text=black] at (0.5,5.5) {4};
\node[text=black] at (0.5,4.5) {5};
\node[text=black] at (0.5,3.5) {6};
\node[text=black] at (0.5,2.5) {7};
\node[text=black] at (0.5,1.5) {8};
\node[text=black] at (1,0.5) {\bf .};
\node[text=black] at (1,0.3) {\bf .};
\node[text=black] at (1,0.1) {\bf .};

\draw[line width=1mm] (1,1) --(1,9) -- (9,9);
\draw[step=1cm,gray,very thin] (1,8) grid (8,9);
\draw[step=1cm,gray,very thin] (1,7) grid (7,8);
\draw[step=1cm,gray,very thin] (1,6) grid (6,7);
\draw[step=1cm,gray,very thin] (1,5) grid (5,6);
\draw[step=1cm,gray,very thin] (1,4) grid (4,5);
\draw[step=1cm,gray,very thin] (1,3) grid (3,4);
\draw[step=1cm,gray,very thin] (1,2) grid (2,3);

\node[text=black] at (5.7,5.3) {\bf .};
\node[text=black] at (5.9,5.1) {\bf .};
\node[text=black] at (6.1,4.9) {\bf .};

\draw[line width=0.5mm] (1,8.5) -- (2,8.5);
\draw[line width=0.5mm] (1.5,8) -- (1.5,9);
\draw [fill=red] (1.5,8.5) circle [radius=0.1];
\draw[line width=0.5mm] (2,8.5) -- (3,8.5);
\draw[line width=0.5mm] (2.5,8) -- (2.5,9);
\draw [fill=red] (2.5,8.5) circle [radius=0.1];
\draw[line width=0.5mm] (3,8.5) -- (3.5,9);
\draw[line width=0.5mm] (4,8.5) -- (3.5,8);
\draw[line width=0.5mm] (4,8.5) -- (4.5,9);
\draw[line width=0.5mm] (5,8.5) -- (4.5,8);
\draw[line width=0.5mm] (5,8.5) -- (5.5,9);
\draw[line width=0.5mm] (6,8.5) -- (5.5,8);
\draw[line width=0.5mm] (6,8.5) -- (6.5,9);
\draw[line width=0.5mm] (7,8.5) -- (6.5,8);
\draw[line width=0.5mm] (7,8.5) -- (7.5,9);
\draw[line width=0.5mm] (8,8.5) -- (7.5,8);
\draw[line width=0.5mm] (8,8.5) -- (8.5,9);
\draw[line width=0.5mm] (1,7.5) -- (2,7.5);
\draw[line width=0.5mm] (1.5,7) -- (1.5,8);
\draw [fill=red] (1.5,7.5) circle [radius=0.1];
\draw[line width=0.5mm] (2,7.5) -- (2.5,8);
\draw[line width=0.5mm] (3,7.5) -- (2.5,7);
\draw[line width=0.5mm] (3,7.5) -- (3.5,8);
\draw[line width=0.5mm] (4,7.5) -- (3.5,7);
\draw[line width=0.5mm] (4,7.5) -- (5,7.5);
\draw[line width=0.5mm] (4.5,7) -- (4.5,8);
\draw [fill=red] (4.5,7.5) circle [radius=0.1];
\draw[line width=0.5mm] (5,7.5) -- (5.5,8);
\draw[line width=0.5mm] (6,7.5) -- (5.5,7);
\draw[line width=0.5mm] (6,7.5) -- (6.5,8);
\draw[line width=0.5mm] (7,7.5) -- (6.5,7);
\draw[line width=0.5mm] (7,7.5) -- (7.5,8);
\draw[line width=0.5mm] (1,6.5) -- (2,6.5);
\draw[line width=0.5mm] (1.5,6) -- (1.5,7);
\draw [fill=red] (1.5,6.5) circle [radius=0.1];
\draw[line width=0.5mm] (2,6.5) -- (3,6.5);
\draw[line width=0.5mm] (2.5,6) -- (2.5,7);
\draw [fill=red] (2.5,6.5) circle [radius=0.1];
\draw[line width=0.5mm] (3,6.5) -- (3.5,7);
\draw[line width=0.5mm] (4,6.5) -- (3.5,6);
\draw[line width=0.5mm] (4,6.5) -- (4.5,7);
\draw[line width=0.5mm] (5,6.5) -- (4.5,6);
\draw[line width=0.5mm] (5,6.5) -- (5.5,7);
\draw[line width=0.5mm] (6,6.5) -- (5.5,6);
\draw[line width=0.5mm] (6,6.5) -- (6.5,7);
\draw[line width=0.5mm] (1,5.5) -- (1.5,6);
\draw[line width=0.5mm] (2,5.5) -- (1.5,5);
\draw[line width=0.5mm] (2,5.5) -- (3,5.5);
\draw[line width=0.5mm] (2.5,5) -- (2.5,6);
\draw [fill=red] (2.5,5.5) circle [radius=0.1];
\draw[line width=0.5mm] (3,5.5) -- (3.5,6);
\draw[line width=0.5mm] (4,5.5) -- (3.5,5);
\draw[line width=0.5mm] (4,5.5) -- (4.5,6);
\draw[line width=0.5mm] (5,5.5) -- (4.5,5);
\draw[line width=0.5mm] (5,5.5) -- (5.5,6);
\draw[line width=0.5mm] (1,4.5) -- (1.5,5);
\draw[line width=0.5mm] (2,4.5) -- (1.5,4);
\draw[line width=0.5mm] (2,4.5) -- (2.5,5);
\draw[line width=0.5mm] (3,4.5) -- (2.5,4);
\draw[line width=0.5mm] (3,4.5) -- (3.5,5);
\draw[line width=0.5mm] (4,4.5) -- (3.5,4);
\draw[line width=0.5mm] (4,4.5) -- (4.5,5);
\draw[line width=0.5mm] (1,3.5) -- (1.5,4);
\draw[line width=0.5mm] (2,3.5) -- (1.5,3);
\draw[line width=0.5mm] (2,3.5) -- (2.5,4);
\draw[line width=0.5mm] (3,3.5) -- (2.5,3);
\draw[line width=0.5mm] (3,3.5) -- (3.5,4);
\draw[line width=0.5mm] (1,2.5) -- (1.5,3);
\draw[line width=0.5mm] (2,2.5) -- (1.5,2);
\draw[line width=0.5mm] (2,2.5) -- (2.5,3);
\draw[line width=0.5mm] (1,1.5) -- (1.5,2);
\node[text=black] at (12.5,9.5) {degrees of $x_i$};
\node[text=red] at (12.5,8.5) {2};
\node[text=red] at (12.5,7.5) {2};
\node[text=red] at (12.5,6.5) {2};
\node[text=red] at (12.5,5.5) {1};
\node[text=red] at (12.5,4.5) {0};
\node[text=red] at (12.5,3.5) {0};
\node[text=red] at (12.5,2.5) {0};
\node[text=red] at (12.5,1.5) {0};
\node[draw] at (15.5,5) {$m(w)=x_1^2x_2^2x_3^2x_4$};
\end{tikzpicture}}}
\caption{An $\RC$-graph for $w=326154789\ldots \in \PermutN$ and the corresponding monomial.}
\label{RC-graph}
\end{figure}
\begin{theorem}[cf.~\cite{FK-rc}] \label{rc-formula}
For any permutation $w\in \PermutN$, its Schubert polynomial is given by
$$\Sch_w=\sum_{g\in \RC(w)}m(g).$$
\end{theorem}

Recently, another combinatorial description of the Schubert polynomial~\cite{LLS} was obtained: bumpless pipe dream can be seen as a generalization of Rothe diagrams. Until now there is no combinatorial bijection between the two descriptions.

Note that for some reduced decomposition $h$, its set $C(h)$ may be empty. It is easier to work with the reduced decompositions when all $C(h)$ are non-empty. To ensure this, we need to consider shifts of permutations. 
Define the shift $\shift$ on permutations $\shift: \PermutN \to \PermutN$ as
$$\shift w(1):=1\ \textrm{and} \ \shift w (i+1):=w(i)+1,\ \textrm{for}\ w\in \PermutN.$$
The ``toward'' shift was considered, for example, in~\cite{BB, EG, St-sym}. Unfortunately, the problem is that this ``limit'' is a symmetric function which does not ``remember" the initial Schubert polynomial. Nevertheless, it still seems very important and useful. 
\begin{definition}[cf.~\cite{St-sym}]
For a permutation $w\in \PermutN$, define the Stanley symmetric function as the formal expression 
$$\F_w=\F_w(x_i,i\in \bN):=\lim _{k\to+\infty} \Sch_{\shift^kw}(x_{1},x_{2},x_{3},\ldots)\in \bZ[x_i,\ i\in \bN].$$
\end{definition}
Since $\F_w$ is a symmetric function, it admits an expression via Schur polynomials. P.\,Edelman and C.\,Greene found  the expression, their main result states that the coefficients in this expression are non-negative.
\begin{theorem}[cf.~\cite{EG}] For any permutation $w\in \PermutN$
$$\F_w=\sum_{\lambda}a_{\lambda,w}s_{\lambda},$$
where $a_{\lambda,w}$ are non-negative.
\end{theorem}
Another stability is considered in the next section. 

\nextsection
\section{Back stable Schubert polynomials}
\label{sec:backstable}

Let us now consider permutations on a larger set  $\bZ$. Let $\PermutZ$ be the set of permutations on $\bZ$ fixing all but a finite number of elements.
Define the shift~$\shift$ on permutations $\PermutZ$ as
$$\shift w (i+1):=w(i)+1,\ \textrm{for}\ w\in \PermutZ.$$  It is more convenient to work with all integers rather than with just the positives,  as it was considered by T.\,Lam, S.\,J.\,Lee, and M.\,Schimozono~\cite{LLS}. We use their notations. Denote by $\Lambda$ the ring of symmetric polynomials in $\{x_i,\ i\leq 0\}$. 

\begin{definition}[cf.~\cite{LLS}]
For a permutation $w\in \PermutZ$ define its back stable Schubert polynomial as the formal expression:
$$\BSch_w=\BSch_w(x_i,i\in \bZ):=\lim _{k\to+\infty} \Sch_{\shift^kw}(x_{1-k},x_{2-k},x_{3-k},\ldots)\in \Lambda \otimes \bQ[x_i,\ i\in \bZ].$$
\end{definition}
\begin{remark}
For permutations $w\in \PermutN$, we have the following relation between their back stable and their usual Schubert polynomials:
$$\BSch_w(x_i=0, i\leq 0)=\Sch_w.$$
\end{remark}

\begin{proposition}
The definition is correct and $\{\BSch_w,\ w\in\PermutZ\}$ are linearly independent.
\end{proposition}
\begin{proof}
The first part is clear from the RC-graphs constructions and from the reduced decomposition formula. The second part holds, since any two back stable polynomials have different leading monomials in lexicographic order with the alphabet $$\ldots<x_{-3}<x_{-2}<x_{-1}<x_0<x_2<x_3<\ldots.$$
More specifically, for a permutation $w\in \PermutZ$, 
its leading monomial equals $\prod_{i=-\infty}^{\infty}x_i^{d_i},$ 
where $(\ldots,d_{-2},d_{-1},d_{0},d_{1},d_{2},\ldots)$ is the Lehmer code of $w$ (this is easy to see from Theorem~\ref{rc-formula}).
\end{proof}

\begin{proposition}[cf.~\cite{LLS}]
Given a pair of permutations $u,v\in \PermutZ,$ there is a  unique set of constants $c_{u,v}^{\bullet}$ such that
$$\BSch_u\BSch_v=\sum_{w\in \PermutZ} c_{u,v}^w \BSch_w.$$
\end{proposition}

For a triplet of permutations $u,v,w\in \PermutN,$ the structure constant $c_{u,v}^w$ is exactly the structure constant for the usual Schubert polynomials. Furthermore, we do not ``get'' new constants, because all the new constants are equal to the structure constants for the original Schubert polynomials.  Namely, the structure constants for back stable Schubert polynomials satisfy the following relations:
$$c_{u,v}^w=c_{\shift^k u,\tau^k v}^{\tau^k w},\ \text{for}\ u,v,w\in \PermutZ\ \text{and}\ k\in \bZ,$$
which implies that any back stable structure constant is equal to some original constant for large $k$.

Although we know that finding these constants for $\BSch$ is equivalent to finding them for $\Sch$, back stable Schubert polynomials have a few more properties.
\begin{proposition} \label{pr:equality} Given a pair of permutations $u,v\in \PermutZ$, the following holds:
$$\mybinom{\ell(u)+\ell(v)}{\ell(v)}|\R(u)||\R(v)|=\sum_{w\in \PermutZ} c_{u,v}^w\ |\R(w)|,$$
where $\ell(u)$ is the length (the number of inversions) of $u$ and $\R(u)$ is the set of its reduced words.
\end{proposition}
\begin{proof} Note that for any permutation $u$ and a sufficient large number $N$, we have that the coefficient of any monomial $x_{i_1}\ldots x_{i_{\ell(u)}},\  (-N)>i_1>\ldots >i_{\ell(u)}$ is equal to the number of the reduced decompositions of $u$. Therefore, we obtain our equality. 
\end{proof}

The above proposition gives some hope that, for back stable Schubert polynomials, one can construct a rule by ``merging'' the reduced decompositions, see fig.~\ref{fig:merge} ($\BSch_{(01324)}\BSch_{(02314)}=\BSch_{(12304)}+\BSch_{(02413)}$; the reduced decompositions are drawn as wiring diagrams). 
Since all structure constants are positive, this procedure should exist. Furthermore there is a procedure, which agrees with $\mydiff$ (see the next subsection).  
It is possible to construct such a rule for the product of Schur functions using Edelman-Greene's algorithm~\cite{EG} (the algorithm was introduced originally to express Stanley symmetric functions~\cite{St-sym} in terms of Schur functions).

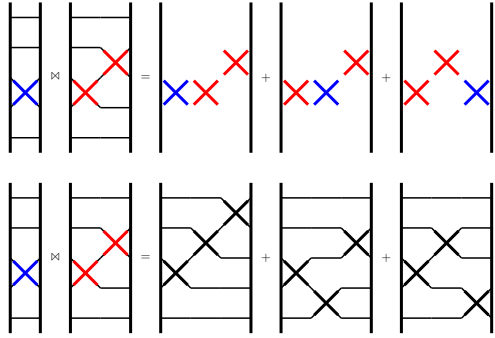
\begin{figure}[htb!]
\scalebox{0.4}{
\hidefigure{
\begin{tikzpicture}

\draw[line width=1mm] (-0.5,-0.5) --(-0.5,4.5);
\draw[line width=0.5mm] (-0.5,1) -- (0.5,2);
\draw[line width=0.5mm] (-0.5,2) -- (0.5,1);
\draw[line width=1mm, blue] (-0.4,1.1) -- (0.4,1.9);
\draw[line width=1mm, blue] (-0.4,1.9) -- (0.4,1.1);
\draw[line width=0.5mm] (0.5,4) -- (-0.5,4);
\draw[line width=0.5mm] (0.5,3) -- (-0.5,3);
\draw[line width=0.5mm] (0.5,0) -- (-0.5,0);
\draw[line width=1mm] (0.5,-0.5) --(0.5,4.5);
\node[text=black] at (1,2) {$\Join$};
\draw[line width=1mm] (1.5,-0.5) --(1.5,4.5);
\draw[line width=0.5mm] (1.5,1) -- (2.5,2);
\draw[line width=0.5mm] (1.5,2) -- (2.5,1);
\draw[line width=1mm, red] (1.6,1.1) -- (2.4,1.9);
\draw[line width=1mm, red] (1.6,1.9) -- (2.4,1.1);
\draw[line width=0.5mm] (2.5,4) -- (1.5,4);
\draw[line width=0.5mm] (2.5,3) -- (1.5,3);
\draw[line width=0.5mm] (2.5,0) -- (1.5,0);
\draw[line width=0.5mm] (2.5,2) -- (3.5,3);
\draw[line width=0.5mm] (2.5,3) -- (3.5,2);
\draw[line width=1mm, red] (2.6,2.1) -- (3.4,2.9);
\draw[line width=1mm, red] (2.6,2.9) -- (3.4,2.1);
\draw[line width=0.5mm] (3.5,4) -- (2.5,4);
\draw[line width=0.5mm] (3.5,1) -- (2.5,1);
\draw[line width=0.5mm] (3.5,0) -- (2.5,0);
\draw[line width=1mm] (3.5,-0.5) --(3.5,4.5);
\node[text=black] at (4,2) {=};
\draw[line width=1mm] (4.5,-0.5) --(4.5,4.5);
\draw[line width=1mm, blue] (4.6,1.1) -- (5.4,1.9);
\draw[line width=1mm, blue] (4.6,1.9) -- (5.4,1.1);
\draw[line width=1mm, red] (5.6,1.1) -- (6.4,1.9);
\draw[line width=1mm, red] (5.6,1.9) -- (6.4,1.1);
\draw[line width=1mm, red] (6.6,2.1) -- (7.4,2.9);
\draw[line width=1mm, red] (6.6,2.9) -- (7.4,2.1);
\draw[line width=1mm] (7.5,-0.5) --(7.5,4.5);
\node[text=black] at (8,2) {+};
\draw[line width=1mm] (8.5,-0.5) --(8.5,4.5);
\draw[line width=1mm, red] (8.6,1.1) -- (9.4,1.9);
\draw[line width=1mm, red] (8.6,1.9) -- (9.4,1.1);
\draw[line width=1mm, blue] (9.6,1.1) -- (10.4,1.9);
\draw[line width=1mm, blue] (9.6,1.9) -- (10.4,1.1);
\draw[line width=1mm, red] (10.6,2.1) -- (11.4,2.9);
\draw[line width=1mm, red] (10.6,2.9) -- (11.4,2.1);
\draw[line width=1mm] (11.5,-0.5) --(11.5,4.5);
\node[text=black] at (12,2) {+};
\draw[line width=1mm] (12.5,-0.5) --(12.5,4.5);
\draw[line width=1mm, red] (12.6,1.1) -- (13.4,1.9);
\draw[line width=1mm, red] (12.6,1.9) -- (13.4,1.1);
\draw[line width=1mm, red] (13.6,2.1) -- (14.4,2.9);
\draw[line width=1mm, red] (13.6,2.9) -- (14.4,2.1);
\draw[line width=1mm, blue] (14.6,1.1) -- (15.4,1.9);
\draw[line width=1mm, blue] (14.6,1.9) -- (15.4,1.1);
\draw[line width=1mm] (15.5,-0.5) --(15.5,4.5);
\draw[line width=1mm] (-0.5,-6.5) --(-0.5,-1.5);
\draw[line width=0.5mm] (-0.5,-5) -- (0.5,-4);
\draw[line width=0.5mm] (-0.5,-4) -- (0.5,-5);
\draw[line width=1mm, blue] (-0.4,-4.9) -- (0.4,-4.1);
\draw[line width=1mm, blue] (-0.4,-4.1) -- (0.4,-4.9);
\draw[line width=0.5mm] (0.5,-2) -- (-0.5,-2);
\draw[line width=0.5mm] (0.5,-3) -- (-0.5,-3);
\draw[line width=0.5mm] (0.5,-6) -- (-0.5,-6);
\draw[line width=1mm] (0.5,-6.5) --(0.5,-1.5);
\node[text=black] at (1,-4) {$\Join$};
\draw[line width=1mm] (1.5,-6.5) --(1.5,-1.5);
\draw[line width=0.5mm] (1.5,-5) -- (2.5,-4);
\draw[line width=0.5mm] (1.5,-4) -- (2.5,-5);
\draw[line width=1mm, red] (1.6,-4.9) -- (2.4,-4.1);
\draw[line width=1mm, red] (1.6,-4.1) -- (2.4,-4.9);
\draw[line width=0.5mm] (2.5,-2) -- (1.5,-2);
\draw[line width=0.5mm] (2.5,-3) -- (1.5,-3);
\draw[line width=0.5mm] (2.5,-6) -- (1.5,-6);
\draw[line width=0.5mm] (2.5,-4) -- (3.5,-3);
\draw[line width=0.5mm] (2.5,-3) -- (3.5,-4);
\draw[line width=1mm, red] (2.6,-3.9) -- (3.4,-3.1);
\draw[line width=1mm, red] (2.6,-3.1) -- (3.4,-3.9);
\draw[line width=0.5mm] (3.5,-2) -- (2.5,-2);
\draw[line width=0.5mm] (3.5,-5) -- (2.5,-5);
\draw[line width=0.5mm] (3.5,-6) -- (2.5,-6);
\draw[line width=1mm] (3.5,-6.5) --(3.5,-1.5);
\node[text=black] at (4,-4) {=};
\draw[line width=1mm] (4.5,-6.5) --(4.5,-1.5);
\draw[line width=0.5mm] (4.5,-5) -- (5.5,-4);
\draw[line width=0.5mm] (4.5,-4) -- (5.5,-5);
\draw[line width=1mm, black] (4.6,-4.9) -- (5.4,-4.1);
\draw[line width=1mm, black] (4.6,-4.1) -- (5.4,-4.9);
\draw[line width=0.5mm] (5.5,-2) -- (4.5,-2);
\draw[line width=0.5mm] (5.5,-3) -- (4.5,-3);
\draw[line width=0.5mm] (5.5,-6) -- (4.5,-6);
\draw[line width=0.5mm] (5.5,-4) -- (6.5,-3);
\draw[line width=0.5mm] (5.5,-3) -- (6.5,-4);
\draw[line width=1mm, black] (5.6,-3.9) -- (6.4,-3.1);
\draw[line width=1mm, black] (5.6,-3.1) -- (6.4,-3.9);
\draw[line width=0.5mm] (6.5,-2) -- (5.5,-2);
\draw[line width=0.5mm] (6.5,-5) -- (5.5,-5);
\draw[line width=0.5mm] (6.5,-6) -- (5.5,-6);
\draw[line width=0.5mm] (6.5,-3) -- (7.5,-2);
\draw[line width=0.5mm] (6.5,-2) -- (7.5,-3);
\draw[line width=1mm, black] (6.6,-2.9) -- (7.4,-2.1);
\draw[line width=1mm, black] (6.6,-2.1) -- (7.4,-2.9);
\draw[line width=0.5mm] (7.5,-4) -- (6.5,-4);
\draw[line width=0.5mm] (7.5,-5) -- (6.5,-5);
\draw[line width=0.5mm] (7.5,-6) -- (6.5,-6);
\draw[line width=1mm] (7.5,-6.5) --(7.5,-1.5);
\node[text=black] at (8,-4) {+};
\draw[line width=1mm] (8.5,-6.5) --(8.5,-1.5);
\draw[line width=0.5mm] (8.5,-5) -- (9.5,-4);
\draw[line width=0.5mm] (8.5,-4) -- (9.5,-5);
\draw[line width=1mm, black] (8.6,-4.9) -- (9.4,-4.1);
\draw[line width=1mm, black] (8.6,-4.1) -- (9.4,-4.9);
\draw[line width=0.5mm] (9.5,-2) -- (8.5,-2);
\draw[line width=0.5mm] (9.5,-3) -- (8.5,-3);
\draw[line width=0.5mm] (9.5,-6) -- (8.5,-6);
\draw[line width=0.5mm] (9.5,-6) -- (10.5,-5);
\draw[line width=0.5mm] (9.5,-5) -- (10.5,-6);
\draw[line width=1mm, black] (9.6,-5.9) -- (10.4,-5.1);
\draw[line width=1mm, black] (9.6,-5.1) -- (10.4,-5.9);
\draw[line width=0.5mm] (10.5,-2) -- (9.5,-2);
\draw[line width=0.5mm] (10.5,-3) -- (9.5,-3);
\draw[line width=0.5mm] (10.5,-4) -- (9.5,-4);
\draw[line width=0.5mm] (10.5,-4) -- (11.5,-3);
\draw[line width=0.5mm] (10.5,-3) -- (11.5,-4);
\draw[line width=1mm, black] (10.6,-3.9) -- (11.4,-3.1);
\draw[line width=1mm, black] (10.6,-3.1) -- (11.4,-3.9);
\draw[line width=0.5mm] (11.5,-2) -- (10.5,-2);
\draw[line width=0.5mm] (11.5,-5) -- (10.5,-5);
\draw[line width=0.5mm] (11.5,-6) -- (10.5,-6);
\draw[line width=1mm] (11.5,-6.5) --(11.5,-1.5);
\node[text=black] at (12,-4) {+};
\draw[line width=1mm] (12.5,-6.5) --(12.5,-1.5);
\draw[line width=0.5mm] (12.5,-5) -- (13.5,-4);
\draw[line width=0.5mm] (12.5,-4) -- (13.5,-5);
\draw[line width=1mm, black] (12.6,-4.9) -- (13.4,-4.1);
\draw[line width=1mm, black] (12.6,-4.1) -- (13.4,-4.9);
\draw[line width=0.5mm] (13.5,-2) -- (12.5,-2);
\draw[line width=0.5mm] (13.5,-3) -- (12.5,-3);
\draw[line width=0.5mm] (13.5,-6) -- (12.5,-6);
\draw[line width=0.5mm] (13.5,-4) -- (14.5,-3);
\draw[line width=0.5mm] (13.5,-3) -- (14.5,-4);
\draw[line width=1mm, black] (13.6,-3.9) -- (14.4,-3.1);
\draw[line width=1mm, black] (13.6,-3.1) -- (14.4,-3.9);
\draw[line width=0.5mm] (14.5,-2) -- (13.5,-2);
\draw[line width=0.5mm] (14.5,-5) -- (13.5,-5);
\draw[line width=0.5mm] (14.5,-6) -- (13.5,-6);
\draw[line width=0.5mm] (14.5,-6) -- (15.5,-5);
\draw[line width=0.5mm] (14.5,-5) -- (15.5,-6);
\draw[line width=1mm, black] (14.6,-5.9) -- (15.4,-5.1);
\draw[line width=1mm, black] (14.6,-5.1) -- (15.4,-5.9);
\draw[line width=0.5mm] (15.5,-2) -- (14.5,-2);
\draw[line width=0.5mm] (15.5,-3) -- (14.5,-3);
\draw[line width=0.5mm] (15.5,-4) -- (14.5,-4);
\draw[line width=1mm] (15.5,-6.5) --(15.5,-1.5);

\end{tikzpicture}}}
\caption{Merge of reduced decompositions}
\label{fig:merge}
\end{figure}

\nextsubsection
\subsection{Differential operators $\mydiff$ and $\nabla$}
\label{sec:diff}

Using an approach similar to Proposition~\ref{pr:equality}, we can construct a differential operator for a back stable Schubert polynomial.
Define the operator
 $\mydiff:\ \Lambda \otimes \bZ[x_i,\ i\in \bZ]\to \Lambda \otimes \bZ[x_i,\ i\in \bZ]$ as follows:
$$\mydiff(f)= \sum_{\gamma\in \bZ_{\geq0}^\bZ} (\lim_{k\to -\infty}\text{coef. of $x^{\gamma}x_k$ in $f$}) \cdot  x^{\gamma}.$$
It is easy to see that $\mydiff$ satisfies the Leibniz rule, i.e., $\mydiff(fg)=(\mydiff f)g+f(\mydiff g)$. We can write the explicit formula for this operator action on a back stable Schubert polynomial.

\begin{proposition} The operator $\mydiff$ satisfies the following formula:
$$\mydiff \BSch_u=\sum_{k:\ \ell(s_ku)=\ell(u)-1}\BSch_{s_ku}.$$
Furthermore, the operator satisfies the Leibniz rule, i.e., 
$$\mydiff(\BSch_u\BSch_v)=(\mydiff\BSch_u)\BSch_v+\BSch_u(\mydiff\BSch_v).$$
\end{proposition}
\begin{proof}
As it was mentioned above, $\mydiff$ satisfies the product formula.
The first part of the Theorem immediately follows from the definition of back stable Schubert polynomials together with  Theorem~\ref{stanley-formula}. 
\end{proof}

For the original Schubert polynomials, another differential operator was constructed by R.\,Stanley~\cite{St-nabla}, see also~\cite{HPSW}, 
which can be easily extended for $\BSch$. 
Define the operator $\nabla$ as follows:
$$\nabla \BSch_u=\sum_{k:\ \ell(s_ku):=\ell(u)-1}k \BSch_{s_ku}.$$

\begin{proposition} The operator $\nabla$ satisfies the Leibniz rule, i.e., 
$$\nabla(\BSch_u\BSch_v)=(\nabla\BSch_u)\BSch_v+\BSch_u(\nabla\BSch_v).$$
\end{proposition}
\begin{proof}
Fix a pair of permutations $u, v\in \PermutZ$ and choose a sufficiently large $k$. We can make $k$ shifts $\shift$ and use $\nabla$ for regular Schubert polynomials and then do again $k$ shifts $\shift^{-1}$. Therefore, the operator $\nabla+k\mydiff$ satisfies the Leibniz rule, and hence so does $\nabla$. 
\end{proof}

As me mention above, one can do a shift $\shift$ and get a new operator $\widetilde{\nabla}=\nabla+\mydiff$. In fact, all such linear operators are linear combinations of $\mydiff$ and $\nabla$.

\begin{theorem} If an operator $\zeta$ satisfies:
\begin{enumerate}
\item $\zeta \BSch_u=\sum_{k:\ \ell(s_ku)=\ell(u)-1}a_{u,k}\BSch_{s_ku},\ a_{u,k}\in \bQ$;
\item $\zeta(\BSch_u\BSch_v)=(\zeta\BSch_u)\BSch_v+\BSch_u(\zeta\BSch_v)$,
\end{enumerate}
then $\zeta$ is a linear combination of $\mydiff$ and $\nabla$.
\end{theorem}
The proof is based on section~\S4, which we suggest to read before.
\begin{proof}
Any differential operator on Grassmannian permutations can be expressed as a linear combination of $\mydiff$ and $\nabla$, see Theorem~\ref{schur:lin}. Therefore without loss of generality,
we can assume that $\zeta(\BSch(w))=0$ for any Grassmannian permutation $w$ of descent $0$.

We will use the following equation: $$\BSch_{[1,(k)]}\BSch_{[0,(1)]}=\BSch_{[0,(k+1)]}+\BSch_{[1,(k,1)]},$$
where $[k,(\lambda)]$ is a Grassmannian permutation of descent $k$ corresponding to $\lambda$.
We have $\zeta(\BSch_{[1,(k)]})=a\BSch_{[1,(k-1)]}$ for some rational $a$, hence,
$$\zeta(\BSch_{[1,(k)]}\BSch_{[0,(1)]})=a\BSch_{[0,(k)]}+a\BSch_{[1,(k-1,1)]}.$$
This gives us
$$\zeta(\BSch_{[1,(k,1)]})=a\BSch_{[0,(k)]}+a\BSch_{[1,(k-1,1)]},$$
whence $a=0$. Since $\BSch_{[1,(k)]}$ forms a multiplicative basis for the subring generated by Grassmannian permutations of descent $1$, we get that $\zeta(\BSch(w))=0$ for any Grassmanian permutation $w$ of descent $1$. Similarly, we can show the same holds for all positive and negative descents. 
We get that $\zeta(\BSch(w))=0$ for any Grassmanian permutation, and thus for all of them.
\end{proof}

\nextsection
\section{New point of view on Schur polynomials}
\label{sec:newpoint}

Here we will work with Schur functions as symmetric polynomials in $\{x_i:\ i\in \bZ_{\leq 0}\};$ the  assumption of $i$ being non-positive allows us to keep the notations consistent in the rest of the paper. 
\nextsubsection
\subsection{From Schubert to Schur}  A descent (ascent) of $w\in \PermutZ$ is a position $i\in Z$ with $w(i)>w(i+1)$ ($w(i)<w(i+1)$). It is well known that Schubert polynomials are symmetric in $x_i$ and $x_{i+1}$ if and only if $i$ is  an ascent (and thus, the same holds for back stable Schubert polynomials). 

A permutation is a Grassmannian permutation if and only if it has at most one descent.  
It is easy to see that if $k$ is a unique descent of $w\in \PermutZ$, then $\BSch$ is a symmetric function in $\{x_i,\ i\leq k\}$; they correspond to Schur functions. A Grassmannian permutation $u\in \PermutZ$ of descent $k$ determines a Young diagram
$$\lambda(u)=(u_k-k, u_{k-1}-k+1,u_{k-2}-k+2,\ldots).$$
We have the following equality 
 $$\Schur_{\lambda(u)}(x_i, i\in(-\infty,k])=\BSch_u.$$
 There are different definitions of Schur polynomials, see~\cite{Ful}. 
 
\begin{remark}\label{red-to-Schur} We can easily construct the Young diagram corresponding to a Grassmannian permutation. At first, we consider the wiring diagram of any reduced decomposition of a Grassmannian permutation. Then we rotate and mirror the picture. All the intersections together form the Young diagram (see fig.~\ref{young-inter}).
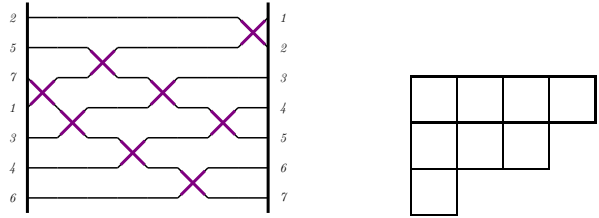
\begin{figure}[htb!]
\centering
\begin{tabular}{cc}
	\scalebox{0.4}{\hidefigure{
\begin{tikzpicture}

\draw[line width=1mm] (0.5,-0.5) --(0.5,6.5);
\node[text=black] at (0,6) {2};
\node[text=black] at (0,5) {5};
\node[text=black] at (0,4) {7};
\node[text=black] at (0,3) {1};
\node[text=black] at (0,2) {3};
\node[text=black] at (0,1) {4};
\node[text=black] at (0,0) {6};
\draw[line width=1mm] (8.5,-0.5) --(8.5,6.5);
\node[text=black] at (9,6) {1};
\node[text=black] at (9,5) {2};
\node[text=black] at (9,4) {3};
\node[text=black] at (9,3) {4};
\node[text=black] at (9,2) {5};
\node[text=black] at (9,1) {6};
\node[text=black] at (9,0) {7};
\draw[line width=0.5mm] (7.5,5) -- (8.5,6);
\draw[line width=0.5mm] (7.5,6) -- (8.5,5);
\draw[line width=1mm, violet] (7.6,5.1) -- (8.4,5.9);
\draw[line width=1mm, violet] (7.6,5.9) -- (8.4,5.1);
\draw[line width=0.5mm] (8.5,4) -- (7.5,4);
\draw[line width=0.5mm] (8.5,3) -- (7.5,3);
\draw[line width=0.5mm] (8.5,2) -- (7.5,2);
\draw[line width=0.5mm] (8.5,1) -- (7.5,1);
\draw[line width=0.5mm] (8.5,0) -- (7.5,0);
\draw[line width=0.5mm] (6.5,2) -- (7.5,3);
\draw[line width=0.5mm] (6.5,3) -- (7.5,2);
\draw[line width=1mm, violet] (6.6,2.1) -- (7.4,2.9);
\draw[line width=1mm, violet] (6.6,2.9) -- (7.4,2.1);
\draw[line width=0.5mm] (7.5,6) -- (6.5,6);
\draw[line width=0.5mm] (7.5,5) -- (6.5,5);
\draw[line width=0.5mm] (7.5,4) -- (6.5,4);
\draw[line width=0.5mm] (7.5,1) -- (6.5,1);
\draw[line width=0.5mm] (7.5,0) -- (6.5,0);
\draw[line width=0.5mm] (5.5,0) -- (6.5,1);
\draw[line width=0.5mm] (5.5,1) -- (6.5,0);
\draw[line width=1mm, violet] (5.6,0.1) -- (6.4,0.9);
\draw[line width=1mm, violet] (5.6,0.9) -- (6.4,0.1);
\draw[line width=0.5mm] (6.5,6) -- (5.5,6);
\draw[line width=0.5mm] (6.5,5) -- (5.5,5);
\draw[line width=0.5mm] (6.5,4) -- (5.5,4);
\draw[line width=0.5mm] (6.5,3) -- (5.5,3);
\draw[line width=0.5mm] (6.5,2) -- (5.5,2);
\draw[line width=0.5mm] (4.5,3) -- (5.5,4);
\draw[line width=0.5mm] (4.5,4) -- (5.5,3);
\draw[line width=1mm, violet] (4.6,3.1) -- (5.4,3.9);
\draw[line width=1mm, violet] (4.6,3.9) -- (5.4,3.1);
\draw[line width=0.5mm] (5.5,6) -- (4.5,6);
\draw[line width=0.5mm] (5.5,5) -- (4.5,5);
\draw[line width=0.5mm] (5.5,2) -- (4.5,2);
\draw[line width=0.5mm] (5.5,1) -- (4.5,1);
\draw[line width=0.5mm] (5.5,0) -- (4.5,0);
\draw[line width=0.5mm] (3.5,1) -- (4.5,2);
\draw[line width=0.5mm] (3.5,2) -- (4.5,1);
\draw[line width=1mm, violet] (3.6,1.1) -- (4.4,1.9);
\draw[line width=1mm, violet] (3.6,1.9) -- (4.4,1.1);
\draw[line width=0.5mm] (4.5,6) -- (3.5,6);
\draw[line width=0.5mm] (4.5,5) -- (3.5,5);
\draw[line width=0.5mm] (4.5,4) -- (3.5,4);
\draw[line width=0.5mm] (4.5,3) -- (3.5,3);
\draw[line width=0.5mm] (4.5,0) -- (3.5,0);
\draw[line width=0.5mm] (2.5,4) -- (3.5,5);
\draw[line width=0.5mm] (2.5,5) -- (3.5,4);
\draw[line width=1mm, violet] (2.6,4.1) -- (3.4,4.9);
\draw[line width=1mm, violet] (2.6,4.9) -- (3.4,4.1);
\draw[line width=0.5mm] (3.5,6) -- (2.5,6);
\draw[line width=0.5mm] (3.5,3) -- (2.5,3);
\draw[line width=0.5mm] (3.5,2) -- (2.5,2);
\draw[line width=0.5mm] (3.5,1) -- (2.5,1);
\draw[line width=0.5mm] (3.5,0) -- (2.5,0);
\draw[line width=0.5mm] (1.5,2) -- (2.5,3);
\draw[line width=0.5mm] (1.5,3) -- (2.5,2);
\draw[line width=1mm, violet] (1.6,2.1) -- (2.4,2.9);
\draw[line width=1mm, violet] (1.6,2.9) -- (2.4,2.1);
\draw[line width=0.5mm] (2.5,6) -- (1.5,6);
\draw[line width=0.5mm] (2.5,5) -- (1.5,5);
\draw[line width=0.5mm] (2.5,4) -- (1.5,4);
\draw[line width=0.5mm] (2.5,1) -- (1.5,1);
\draw[line width=0.5mm] (2.5,0) -- (1.5,0);
\draw[line width=0.5mm] (0.5,3) -- (1.5,4);
\draw[line width=0.5mm] (0.5,4) -- (1.5,3);
\draw[line width=1mm, violet] (0.6,3.1) -- (1.4,3.9);
\draw[line width=1mm, violet] (0.6,3.9) -- (1.4,3.1);
\draw[line width=0.5mm] (1.5,6) -- (0.5,6);
\draw[line width=0.5mm] (1.5,5) -- (0.5,5);
\draw[line width=0.5mm] (1.5,2) -- (0.5,2);
\draw[line width=0.5mm] (1.5,1) -- (0.5,1);
\draw[line width=0.5mm] (1.5,0) -- (0.5,0);
\end{tikzpicture}}}\ \ \ & 
	\ \ \ \ \ \ \scalebox{1.3}{$$\yng(4,3,1)$$ }
	\end{tabular}
\label{young-inter}
\caption{A reduced decomposition of $(2571346)\in \PermutZ$ and the corresponding Young diagram $(4,3,1)$.}
\end{figure}
\end{remark}

The ring $\Lambda$ consists of all symmetric functions in $\{x_i,\ i\leq 0\}$. Thus, $\Lambda$ is generated by a Grassmannian permutation of $0$ descent.
Since $\mydiff$ and $\nabla$ cannot add a new descent, we have that $\mydiff$ and $\nabla$ can be restricted to $\Lambda$. Now we will rewrite differential operators in terms of Young diagrams.

\begin{theorem}\label{operators:to:schur}For any $\lambda\in\Diag$ and $k\in \bZ,$ we have
$$\mydiff(s_\lambda(x_i,i\in(-\infty,k]))=\sum_{\substack{(i,j)\in \bN^2\\ \lambda'=\lambda-(i,j)\in \Diag}} s_{\lambda'}(x_i,i\in(-\infty,k]);$$
and
$$\nabla(s_{\lambda}(x_i,i\in(-\infty,k]))=\sum_{\substack{(i,j)\in \bN^2\\ \lambda'=\lambda-(i,j)\in \Diag}} (j-i+k)s_{\lambda'}(x_i,i\in(-\infty,k]),$$
\end{theorem}
\begin{proof}
The proof is trivial by Remark~\ref{red-to-Schur}.
\end{proof}

Decreasing operators $\mydiff$ and $\nabla$ satisfy the Leibniz property. All such differential operators are linear combinations of $\mydiff$ and $\nabla$:
\begin{theorem} \label{schur:lin}
An operator $\zeta$ given by the formula:

 $$\zeta s_\lambda=\sum_{\substack{(i,j)\in \bN^2\\ \lambda'=\lambda-(i,j)\in \Diag}} a_{\lambda,\lambda'}s_{\lambda'},\ a_{\lambda,\lambda'}\in \bQ$$
 satisfies the Leibniz property
 $$\zeta (s_\lambda s_\mu)=(\zeta s_\lambda)s_\mu+s_\lambda(\zeta s_\mu)$$
 if and only if $\zeta$ is a linear combination of $\mydiff$ and $\nabla$.
\end{theorem}
\begin{proof}
We know that $\mydiff$ and $\nabla$ satisfy the Leibniz property. Furthermore, we have 
$$\mydiff s_{\tiny{\yng(2)}}=\nabla s_{\tiny{\yng(1)}}=0,\ \mydiff s_{\tiny{\yng(1)}}=1,\ \textrm{and}\ \nabla s_{\tiny{\yng(2)}}=s_{\tiny{\yng(1)}};$$
hence,
it remains to prove that if $\zeta s_{\tiny{\yng(2)}}=\zeta s_{\tiny{\yng(1)}}=0,$ then $\zeta\equiv 0$. 

We will prove it by induction on the size of diagrams. The base case is $k\leq 2$.
Since  $\zeta(s_{\tiny{\yng(1)}}^2)=\zeta(s_{\tiny{\yng(2)}})+\zeta(s_{\tiny{\yng(1,1)}})$, we obtain: 
$$\zeta(s_{\tiny{\yng(1)}})=\zeta(s_{\tiny{\yng(2)}})=\zeta(s_{\tiny{\yng(1,1)}})=0.$$

Now we prove the induction step.
Let $\zeta(s_{\lambda})=0$ for all diagrams with at most $k$ boxes. It is well known that partitions $s_{(1)},s_{(2)},\ldots $ form a multiplicative basis, so that $s_{(1^{k+1})}$ can be expressed in terms of $s_{(i)}, i\in \bN$. Hence, $\zeta(s_{(1^{k+1})})=as_{(k)},$ where $a\in \bQ$. On the other hand, $\zeta(s_{(1^{k+1})})=bs_{(1^k)},$ where $b\in \bQ$.
We get that $\zeta(s_{(1^{k+1})})=0$, and similarly we have $\zeta(s_{({k+1})})=0$. Since $s_{(i)}$ is a multiplicative basis, we get $\zeta s_{\lambda}=0$ for all diagrams of size $k+1$, which completes the proof.
\end{proof}

\nextsubsection
\subsection{Proof of the key lemma} 
The key lemma immediately follows from the next proposition.
\begin{proposition}\label{alg:prop}
	Given $X=\sum_{i=1}^k a_i\lambda^{(i)},\ 1\neq\lambda^{(i)}\in \Diag$ and $0\leq a_i\in \bQ$, we can recover~$X$ from $\mydiff(X)$ and $\nabla(X)$.
\end{proposition}
\begin{proof}
	Consider the lexicographic order on Young diagrams. Choose the maximal diagram $\mu=(\mu_1\geq\ldots\geq\mu_k\neq 0)$ from $\mydiff(X)$. Since $X$ has positive coefficients, we can obtain $\mu$ only from $\mu'=(\mu_1\geq\ldots\geq\mu_k\geq 1)$ or from $\mu''=(\mu_1\geq\ldots\geq\mu_k+1)$, otherwise $\mydiff(X)$ has a diagram $(\mu_1\geq\ldots\geq\mu_i+1\geq\ldots\geq\mu_k-1)$ with non-zero coefficient for some $i$.  If $\mu''$ is not a Young diagram, then we already know the coefficient of $\mu'$ in $X$. In the second case $\mu$ has distinct coefficients in $\nabla(\mu')$ and $\nabla(\mu'')$ ($k$ and $k-\mu_k-1$ respectively) and has the unit coefficient in $\mydiff(\mu')$ and $\mydiff(\mu'')$. Hence, we can recover both coefficients of $\mu'$ and $\mu''$ in $X$.
	
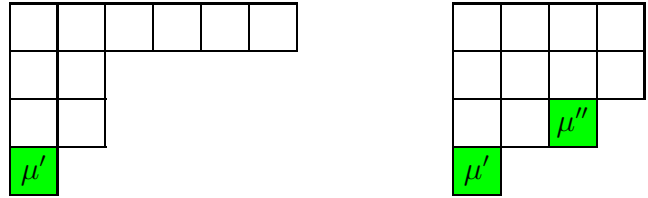
	\begin{figure}[htb!]
\ytableausetup{nosmalltableaux}
  \begin{ytableau}
         \ &   &  &  & &\\
    \ &     \\
      \ & \\
 *(green) \mu'
  \end{ytableau}
  \ \ \ \ \ \  \ \ \ \ \ \ \ \ 
 \begin{ytableau}
\ &   &  &  \\
    \ &    &  &  \\
      \ & & *(green) \mu'' \\
 *(green) \mu'
  \end{ytableau}
\end{figure}
	
	Subtracting these elements from $X$, we get $X'$. The diagrams of $X'$ still have positive coefficients and $\mydiff(X')$ has a smaller leading diagram. Repeat this procedure until we get zero. 
	\end{proof}

The key lemma together with Theorem~\ref{operators:to:schur} prove our first main result.
\begin{corollary}
The ring $\bY$ is well defined and isomorphic to $\Lambda$.
\end{corollary}
Therefore, below we will work with $\bY$ as defined in the introduction. Nevertheless, we will use standard notations for Schur functions.

\nextsubsection
\subsection{Determinantal formulas}
\begin{theorem}[The first Jacobi-Trudi formula]
For $\lambda=(\lambda_1,\ldots, \lambda_k)\in\Diag$, we have
	$$s_{\lambda}=\det\begin{bmatrix}
h_{\lambda_1} & h_{\lambda_1+1} &  h_{\lambda_1+2} &  \dots &  h_{\lambda_1+k-1}\\
    h_{\lambda_2-1} & h_{\lambda_2} &  h_{\lambda_1+1} &  \dots &  h_{\lambda_1+k-2}\\
    \vdots & \vdots & \vdots & \ddots & \vdots\\
     h_{\lambda_k-k+1} & h_{\lambda_k-k+2} &  h_{\lambda_k-k+3} &  \dots &  h_{\lambda_k}
\end{bmatrix}$$
\end{theorem}
\begin{proof}
We prove it by induction on  the size of a digram. The base $|\lambda|=0$ is obvious.

Denote by $\det_{\lambda}:=\det[h_{\lambda_i-i+j}]$ the right hand side of the formula. We have $$\mydiff(h_{\lambda_{i}-i+j})=h_{(\lambda_{i}-1)-i+j};$$
then after combining by rows we get 
$\mydiff(\det_{\lambda})=\sum_{\lambda'=\lambda-(i,\lambda_i)\in \Diag}  \det_{\lambda'}.$

For $\nabla$, we have $$\nabla(h_{\lambda_{i}-i+j})=(\lambda_i-i)h_{(\lambda_{i}-1)-i+j}+(j-1)h_{\lambda_{i}-i+(j-1)}.$$
We combine the left part by rows and 
get $\sum_{\lambda'=\lambda-(i,\lambda_i)\in \Diag}  (i-\lambda_i)\det_{\lambda'}.$
We combine the right part by columns and get $0$, because either two columns $j$ and $j-1$ become proportional, or a factor $j-1=0$.

By induction step the Jacobi-Trudi identity holds for smaller diagrams. Therefore, together with the key lemma, we have
$s_\lambda =\det_{\lambda}$.
\end{proof}

Similarly, we have the second Jacobi-Trudi identity. 
\begin{theorem}[The first Jacobi-Trudi formula]
For $\lambda=(\lambda_1,\ldots, \lambda_k)\in\Diag$, we have
	$$s_{\lambda}=\det\begin{bmatrix}
e_{\lambda'_1} & e_{\lambda'_1+1} &  e_{\lambda'_1+2} &  \dots &  e_{\lambda'_1+k'-1}\\
e_{\lambda'_{k'}-1} & e_{\lambda'_{k'}} &  e_{\lambda'_{k'}+1} &  \dots &  e_{\lambda'_1+k'-2}\\
    \vdots & \vdots & \vdots & \ddots & \vdots\\
e_{\lambda'_{k'}-{k'}+1} & e_{\lambda'_{k'}-{k'}+2} &  e_{\lambda'_{k'}-{k'}+3} &  \dots &  e_{\lambda'_{k'}}
\end{bmatrix},$$
where $\lambda'$ is the conjugate partition to $\lambda$ and $k'$ is the number of columns of $\lambda$.
\end{theorem}

In a similar way one can prove the Giambelli identity; we leave it to the interested readers as an exercise.

\nextsubsection
\subsection{Dual Murnaghana-Nakayama rule and ``Schur'' operators} 

Note that $\mydiff$ and $\nabla$ almost commute. 
In particular, $\mydiff\cdot \nabla -\nabla\cdot \mydiff$ deletes dominos from diagrams (horizontal with coefficient $1$ and vertical with~$-1$). 
We will consider this further.
Define recursively the sequence of differential (bosonic) operators
\begin{itemize}
\item $\rho^{(1)}:=\mydiff$;	
\item $\rho^{(k+1)}:=\frac{[\rho^{(k)},\nabla]}{k}=\frac{\rho^{(k)}\cdot \nabla-\nabla\cdot \rho^{(k)}}{k}.$
\end{itemize}

\begin{proposition} The operators $\rho^{(k)},\ k\in\bN$ satisfy the Leibniz rule, i.e.,
$$\rho^{(k)}(fg)=(\rho^{(k)}f)g+f(\rho^{(k)}g).$$
\end{proposition}
\begin{proof}
We prove it by induction. The base case $k=1$ holds, because $\rho^{(1)}=\mydiff$.
It is easy to check that the commutator of any two operators satisfying the Leibniz rule also satisfies Leibniz rule. Hence, the induction step $k\to k+1$ is clear.
\end{proof}

\begin{theorem} The operators $\rho^{(k)},\ k\in\bN$  are given by
$$\rho^{(k)}s_{\lambda} = \sum_{\mu\subset \lambda, |\mu|=|\lambda|-k} (-1)^{ht(\lambda\setminus \mu) -1}s_{\mu},$$
where the sum ranges over those $\mu$ such that $\lambda\setminus \mu$ is a border strip with $k$
 boxes.
\end{theorem}
\begin{proof}
It is clear that $$\rho^{(k)}s_{\lambda} = \sum_{\mu\subset \lambda, |\mu|=|\lambda|-k} a_{\lambda\setminus \mu}s_{\mu},$$
where the coefficients $a_{\lambda\setminus \mu}$ depend only on $\lambda\setminus \mu$. 

We will prove the original statement by induction; the base case $k=1$ is known. Assume that we know it for $k$.

At first, we prove that $a_{\lambda\setminus \mu}=0$ if $\lambda\setminus \mu$ is disconnected.  We have
 $\rho^{(k+1)}=\frac{\rho^{(k)}\cdot \nabla-\nabla\cdot \rho^{(k)}}{k};$ we can delete something connected by $\rho^{(k)}$ and something connected by $\nabla$. If these two shapes did not touch then we can commute the two operations, which gives us that the difference $\rho^{(k)}\cdot \nabla-\nabla\cdot \rho^{(k)}$ has only connected shapes.
 
The second step is to prove that $a_{\lambda\setminus \mu}=0$ if $\lambda\setminus \mu$ is not a border strip.
By induction step we know that $\rho^{(k+1)}$ can delete only connected shapes, which are almost border strips. More specifically, it is either a border strip or there is a pair with the same content (a row index minus a column index). Assume this content is $d$; then it is easy to see that our shape should contain a $2\times 2$ square . There are two possibilities to split a shape in a border strip and a square; the coefficients for these two splits are the same $(-1)^{ht-1}d$ (in the example, $\nabla$ ``deletes'' the red square). Hence, they contract.

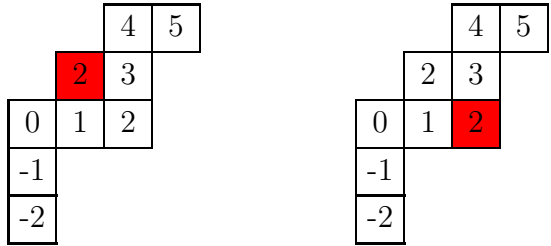
\begin{figure}[htb!]
\ytableausetup{nosmalltableaux}
  \begin{ytableau}
   \none & \none & 4 & 5 \\
   \none &   *(red)2   & 3 \\
0 & 1 & 2\\
      \text{-}1\\
      \text{-}2
  \end{ytableau}
  \ \ \ \ \ \  \ \ \ \ \ \ \ \ 
 \begin{ytableau}
   \none & \none & 4 & 5 \\
   \none &   2   & 3 \\
0 & 1 & *(red)2\\
      \text{-}1\\
      \text{-}2
  \end{ytableau}
\end{figure}

It remains to count the coefficients of the border strips. The operator $\nabla$ should delete the end of a border strip. 
The border strip of length $k+1$ has two ends with the contents $b_1$ and $b_2$ such that $b_1-b_2=k$. 
Consider the case when $\nabla$ deletes the end $b_1$. If $(b_1-1,b_1)$ forms a horizontal domino, then $\nabla$ acts first and then $\rho^{(k)}$; furthermore, the height of the border strip after $\nabla$ remains the same, hence the coefficient is $b_1(-1)^{ht-1}$. If $(b_1-1,b_1)$ forms a vertical domino, then $\nabla$ acts after $\rho^{(k)}$; furthermore, the height will be changed. Hence, the coefficient is again $b_1(-1)^{ht-1}$. Similarly, we get the coefficient for another end to be equal to  $b_2(-1)^{ht}$. The equation $\frac{b_1(-1)^{ht-1}+b_2(-1)^{ht}}{k}=(-1)^{ht-1}$ completes the proof.
\end{proof}

The theorem above is very similar to Murnaghana-Nakayama rule.
\begin{theorem}[Murnaghana-Nakayama rule, cf.~\cite{Mur,Nak}]
 Multiplications by $p_{k},\ k\in\bN$  are given by
$$p_{k}s_{\lambda} = \sum_{\lambda\subset \mu, |\mu|=|\lambda|+k} (-1)^{ht(\mu\setminus \lambda) -1}s_{\mu},$$
where the sum ranges over those $\mu$ such that $\mu\setminus \lambda$ is a border strip with $k$
 boxes.
\end{theorem}

Consider a big rectangle $a\times b$ such that all the discussed diagrams are inside it. If we consider a complement and rotate by $\pi$ (check), the operator $\rho^{(k)}$ acts as multiplication by $p_k$. Therefore, everything is dual and we have a lot of properties similar to $p_k,\ k\in \bN$. 

The operators $\rho^{(k)},\ k\in \bN$ together with $p_k,\ k\in \bN$ are called bosonic operators. They appear in the study of Boson-Fermion correspondence, see~\cite{Lam,Ong}.

\begin{corollary} The operators $\rho^{(k)},\ k\in\bN$ commute pairwise. 
\end{corollary}

Let $\Schurp_{\lambda}(t_1,t_2,\ldots), \lambda\in \Diag$ be Schur functions written in another basis, namely in $\{p_1,p_2,\ldots\}$. Define the operators $\mydiff^{\lambda}:=\Schurp_{\lambda}(\rho^{(1)},\rho^{(2)},\ldots)$.
\begin{proposition}
The operators $\mydiff^{\nu},\ \nu\in\Diag$ act on Schur functions as follows:
 	$$\mydiff^{\nu}s_{\lambda} = \sum_{\mu} c_{\mu,\nu}^{\lambda}s_{\mu},$$
where the $c_{\nu,\mu}^{\lambda}$ are the Littlewood-Richardson coefficients. 

Furthermore,  their product is given by
$$\mydiff^{\nu}\mydiff^{\mu}=\sum_{\lambda}c_{\mu,\nu}^{\lambda}\mydiff^{\lambda}.$$
\end{proposition}

The operators $\mydiff^{\lambda}$ do not satisfy the Leibniz property in general, but we can write how they act on the product.
\begin{proposition}
\label{prop:prod}
 For $X,Y\in\Lambda$ and $\lambda\in\Diag,$ we have
	$$\mydiff^{\lambda}(XY)=\sum_{\mu,\nu}c_{\mu,\nu}^{\lambda}(\mydiff^{\mu}X)(\mydiff^{\nu}Y).$$
\end{proposition}
\begin{proof}
Fix $\lambda$. Since any expression in $\rho^{(k)},k\in\bN$ can be written as a linear combination of $\mydiff^{\mu},\mu\in\Diag$, we get that there is a set of coefficients $b_{\mu,\nu}^{\lambda}$ such that, for any $X,Y,$ we have the equality:
$$\mydiff^{\lambda}(XY)=\sum_{\mu,\nu}b_{\mu,\nu}^{\lambda}(\mydiff^{\mu}X)(\mydiff^{\nu}Y).$$
For $\alpha,\beta\in \Diag$ such that $|\alpha|+|\beta|=|\lambda|,$ we have
$$c_{\alpha,\beta}^{\lambda}=\mydiff^{\lambda}(\sum_{\gamma}c_{\alpha,\beta}^\gamma s_{\gamma})=\mydiff^{\lambda}(s_{\alpha}s_{\beta})=\sum_{\mu,\nu}b_{\mu,\nu}^{\lambda}(\mydiff^{\mu}s_{\alpha})(\mydiff^{\nu}s_{\beta})=b_{\alpha,\beta}^{\lambda},$$
which finishes our proof.
\end{proof}
One more important property of the introduced operators is given by the following proposition:
\begin{proposition} \label{prop:xinabla}We have
$$\mydiff p_1=1;$$ 
$$\mydiff p_k=0,\ k>1;$$
$$\nabla p_1=0;$$
$$\nabla p_k=kp_{k-1},\ k>1.$$
\end{proposition}
\begin{proof}
It is well known that
$$p_k=s_{(k)}-s_{(k-1,1)}+\ldots (-1)^k s_{(1,\ldots,1)}.$$	
If $k=1$ we get $p_{1}=s_{(1)}$; hence, by definition, we get $\mydiff p_1=1$ and $\nabla p_1=0$.

If $k>1$ the proposition immediately follows from
$$\mydiff s_{(a,1,\ldots,1,1)}=s_{(a-1,1,\ldots,1,1)}+s_{(a,1,\ldots,1)}$$
and 
$$\mydiff s_{(a,1,\ldots,1,1)}=(a-1)s_{(a-1,1,\ldots,1,1)}-(k-a)s_{(a,1,\ldots,1)}.$$
\end{proof}

It is possible to check, by direct computation, that:
\begin{proposition}
$$\rho^{(k)}p_{k}=k$$
and
$$\rho^{(k)}p_{k'}=0, \ k'\neq k$$
\end{proposition}
We will not write the proof here, see Lemma~\ref{lem:rho-p} for a more general statement with the same proof.

\begin{remark}This proposition gives us one more connection with characters of the symmetric group. Namely, we already know that $\mydiff^{\lambda}\Schur_{\mu}=\delta_{\lambda,\mu}$ if $|\lambda|=|\mu|$. Furthermore, by the previous proposition, any monomial $\rho^{(i_1)}\ldots \rho^{(i_k)}$ of $\mydiff_{\lambda}$  acts only on the same monomial $p_{i_1}\ldots p_{i_k}$ of $\Schur_{\mu}$ with a certain coefficient.
\end{remark}

\nextsection
\section{Decreasing operators for back stable polynomials}
\label{sec:times}

Similarly to the  previous section, we consider the sequence of differential (bosonic) operators:
\begin{itemize}
\item $\rho^{(1)}:=\mydiff$;	
\item $\rho^{(k+1)}:=\frac{[\rho^{(k)},\nabla]}{k}=\frac{\rho^{(k)}\cdot \nabla-\nabla\cdot \rho^{(k)}}{k}.$
\end{itemize}

Here we have a bigger set of polynomial functions, namely $p_{k,a}=\sum_{i\in (-\infty,a]}x_i^k$. 
\begin{lemma}
\label{lem:rho-p}
$$\rho^{(k)}p_{k,a}=k$$
and
$$\rho^{(k)}p_{k',a}=0, \ k'\neq k$$
\end{lemma}
\begin{proof}
	We will prove it by induction on $k$. Since $\nabla$ and $\mydiff$ act on $p_{k,a}$ similarly to how $\nabla+a\mydiff$ and $\mydiff$ act on $p_{k,0}$, by Proposition~\ref{prop:xinabla} we have: 
\begin{itemize}
\item $\mydiff(p_{1,a})=1$ and $\mydiff(p_{k,a})=0, k>1$; 
\item $\nabla(p_{1,a})=a$ and $\nabla(p_{k,a})=kp_{k-1,a}, k>1$,
\end{itemize}
This gives the base case for $k=1$. Let us check the induction step $k\to k+1$:
$$\rho^{(k+1)}p_{k+1,a}=\frac{\rho^{(k)}\cdot \nabla-\nabla\cdot \rho^{(k)}}{k}p_{k+1,a}=\frac{(k+1)\rho^{(k)}p_{k,a}}{k}=(k+1),$$
and
$$\rho^{(k+1)}p_{k'+1,a}=\frac{\rho^{(k)}\cdot \nabla-\nabla\cdot \rho^{(k)}}{k}p_{k'+1,a}=\frac{(k'+1)\rho^{(k)}p_{k',a}}{k}=0,\ k'> k.$$
If $k'<k$ we immediately get $0$ because the degree after the operator should be $k'-k$.
\end{proof}

\begin{theorem} The operators $\rho^{(k)},\ k\in\bN$ satisfy the Leibniz rule and commute pairwise (for Schubert polynomials). Furthermore, for any $a_1,\ldots,a_k\in \bN$ and $w\in \PermutZ$ such that $\sum a_i=\ell(w),$ we have:
$$\rho^{(a_1)}\rho^{(a_2)}\ldots \rho^{(a_k)}\BSch_w(x_i, i\in \bZ)=\rho^{(a_1)}\rho^{(a_2)}\ldots \rho^{(a_k)}\F_w(x_i, i\in (-\infty,0]).$$
\end{theorem}
\begin{proof}
The first part is trivial. For the second part,  we can express $\BSch_w$ in terms of $p_{k,a}$. Since $\sum a_i=\ell(w)$ and our operators act on $p_{k,a}, a\in \bZ$ identically for all $a$, we can change all $a_i'th$ to $0$. Therefore, we will get $\F_w$ instead of~$\BSch_w$.
\end{proof}
 Since the operators $\rho_k,k\in \bN$ commute, we can again consider the operators $\mydiff^{\lambda},\lambda\in\Diag$. We already know how these operators act on Schur polynomials.
\begin{corollary}\label{cor:operschur}
For a permutation $w$ and a diagram $\lambda$ s.t. $|\lambda|=\ell(w),$ we have
$$\mydiff^{\lambda}\BSch_w=a_{\lambda,w},$$
where the $a_{\lambda,w}$ are the coefficients of the expressions of Stanleys symmetric functions in terms of Schur functions.	
\end{corollary}
\begin{theorem}
For a permutation $w$ and a diagramm $\lambda$, we have
$$\mydiff^{\lambda}\BSch_w=\sum_{\substack{ \ell(u)=|\lambda|\\ \ell(u^{-1}w)=\ell(w)-|\lambda|}}a_{\lambda,u} \BSch_{u^{-1}w},$$
where the $a_{\lambda,u}$ are the coefficients of the expressions of Stanley symmetric functions in terms of Schur functions.
\end{theorem}
\begin{proof}
An operator $\mydiff^{\lambda}$ is an algebraic expression of $\mydiff$ and $\nabla$.
We know that  both $\mydiff$ and $\nabla$ act by simple transposition from the left on a permutation. Hence, if one reduced decomposition of $u$ can act on $w$, then all other decompositions can also act. Therefore the coefficient of $\BSch_{u^{-1}w}$ has two possibilities: is $0$ for some $w$ and is fixed for other permutations. 
We get $$\mydiff^{\lambda}\BSch_w=\sum_{\substack{ \ell(u)=|\lambda|\\ \ell(u^{-1}w)=\ell(w)-|\lambda|}}b_{\lambda,u} \BSch_{u^{-1}w},$$
for some real coefficients $b_{\lambda,u}$.

Substitute a permutation $w$ of length $\ell(w)=|\lambda|.$
By Corollary~\ref{cor:operschur}, $a_{\lambda,w}=b_{\lambda,w},$ which completes the proof.
\end{proof}

We know how the operators $\mydiff^{\lambda}, \lambda\in \Diag$ act on permutations, so we can compute the action of the operators $\rho^{(k)}.$

\begin{theorem}
For any $k\in \bN,$ the differential operator $\rho^{(k)}$ is given by
	$$\rho^{(k)}\BSch_w=\sum_{\substack{  \ell(u^{-1}w)=\ell(w)-k\\(*)}}\pm \BSch_{u^{-1}w},$$
	where the summation (*) is taken over the permutations which admit a reduced word $b_1< b_2<\ldots < b_i>\ldots > b_{k}$ such that any number in the interval $[(min(b_1,b_k),b_i]$  appears at least once. The summands are taken with  signs equal to $(-1)^{k-i}$.
	
\end{theorem}
\begin{proof}
We can reverse the formula and, hence, we have
$$\rho^{(k)}=\mydiff_{(k)}-\mydiff_{(k-1,1)}+\mydiff_{(k-2,1,1)}-\ldots.$$
Each $\mydiff^{\lambda}$  can be written as a result of an application of Edelman-Greene algorithm, see~\cite{EG}. We also know that $\rho^{(k)}$ is a sum over ``connected'' permutations. This finishes our proof.
\end{proof}
Clearly, we still have the Proposition~\ref{prop:prod}, the proof remains without any changes.
\begin{proposition}
 For $X,Y\in\Lambda \otimes \bQ[x_i,\ i\in \bZ]$ and $\lambda\in\Diag,$ we have
	$$\mydiff^{\lambda}(XY)=\sum_{\mu,\nu}c_{\mu,\nu}^{\lambda}(\mydiff^{\mu}X)(\mydiff^{\nu}Y).$$
\end{proposition}

\nextsection
\section{Some remarks about the product of a Schubert polynomial and a Schur polynomial}
\label{sec:times}

In this section we deal with the product of a Schur polynomial and a back stable Schubert polynomial (Schubert times Schur). We assume that Schur polynomials correspond to Grassmannian permutations of descent $0$. We will simply write $\Schur_{\lambda}\times \BSch_w$ and $c_{\lambda,w}^v$.

Here we present a procedure for multiplication of Schubert times Schur. It is well known that we can multiply $s_{\lambda}\times \BSch_w$ using divided difference operators and Stanley symmetric function, however our algorithm is a bit simpler (we do not use the non-trivial Edelman-Greene algorithm). We will multiply recursively by $|\lambda|+\ell(w)$. Denote by $Gr_0$ the set of Grassmannian permutations of descent $0$; then 
 $$\Schur_{\lambda} \BSch_w= \sum_{v\in Gr_0} c_{\lambda,w}^v\BSch_{v}+\sum_{v\notin Gr_0} c_{\lambda,w}^v\BSch_{v}.$$
 Equivalently by Lemma~\ref{lemma:key},
 $$(\mydiff \Schur_{\lambda}) \BSch_w + \Schur_{\lambda}(\mydiff \BSch_w)-\sum_{v\notin Gr_0} c_{\lambda,w}^v \mydiff \BSch_{v}=\mydiff( \sum_{v\in Gr_0} c_{\lambda,w}^v\BSch_{v})$$
  and 
 $$(\nabla \Schur_{\lambda}) \BSch_w + \Schur_{\lambda}(\nabla \BSch_w)-\sum_{v\notin Gr_0} c_{\lambda,w}^v \nabla \BSch_{v}=\nabla( \sum_{v\in Gr_0} c_{\lambda,w}^v\BSch_{v}).$$

By induction we already know that $\sum_{v\notin Gr_0} c_{\lambda,w}^v\BSch_{v},$ because $c_{\lambda,w}^v=c_{\lambda,ws_j}^{vs_j}$ if $j$ is a common descent of $w$ and $v$ (all descents of $w$ except $0$ should be descents of $v$). Therefore we can compute $\sum_{v\notin Gr_0} c_{\lambda,w}^v \mydiff \BSch_{v}$ and $\sum_{v\notin Gr_0} c_{\lambda,w}^v \nabla \BSch_{v}$. Furthermore by induction, we have $(\mydiff \Schur_{\lambda}) \BSch_w + \Schur_{\lambda}(\mydiff \BSch_w)$ and $(\nabla \Schur_{\lambda}) \BSch_w + \Schur_{\lambda}(\nabla \BSch_w)$. 
Hence, we have $\mydiff( \sum_{v\in Gr_0} c_{\lambda,w}^v\BSch_{v})$ and $\nabla( \sum_{v\in Gr_0} c_{\lambda,w}^v\BSch_{v})$. 
By Lemma~\ref{lemma:key} we know that there is a unique $X$ such that $\mydiff(X)=\mydiff( \sum_{v\in Gr_0} c_{\lambda,w}^v\BSch_{v}) \ \textrm{and}\ \nabla(X)=\nabla( \sum_{v\in Gr_0} c_{\lambda,w}^v\BSch_{v})$. 
To optimize the algorithm we can use the fact that $X$ has only positive coefficients; they can be computed using the  algorithm presented in the proof of Proposition~\ref{alg:prop}. 

\begin{remark}
Since we act on permutations by $\partial_i, i\neq 0$, $\mydiff$, and $\nabla$ in our procedure, we consider only permutations, which are smaller in double (right and left) weak Bruhat order.	
\end{remark}

For any conjectured expression for the product of Schur times Schubert, it is enough to check that it agrees with the operators:

\begin{itemize}
 \item{\bf Criterion 1:}\  $\partial_i,i\neq 0,$ $\mydiff,$ and $\nabla$.
 \item{\bf Criterion 2:}\  $\partial_i,i\neq 0$ and $\mydiff^{(i)},i\in \bN$.
 \item{\bf Criterion 3:}\  $\partial_i,i\neq 0$ and $\rho^{(i)},i\in \bN$.
\end{itemize}

\nextsection 
\subsection*{Acknowledgement}  The author is very grateful to Anatol Kirillov for introducing him to the topic. He also would like to thank very much Alex Postnikov, Michael Shapiro, Alejandro Morales, and Laura Colmenarejo for comments and helpful advice while work was in progress. Finally, the author is also grateful to Vasu Tewari for the important references.

   
\end{document}